\documentclass[10pt]{amsart}

\usepackage{times,amsfonts,amsmath,amstext,amsbsy,amssymb,
  amsopn,amsthm,upref,eucal, amscd}
\usepackage[T1]{fontenc}
\usepackage{color}
\usepackage[pdftex]{graphicx}

\newtheorem{theorem}{Theorem}[section]
\newtheorem{lemma}[theorem]{Lemma}
\newtheorem{corollary}[theorem]{Corollary}
\newtheorem{proposition}[theorem]{Proposition}

\numberwithin{equation}{section}

\theoremstyle{definition}
\newtheorem{definition}[theorem]{Definition}

\newtheorem{remark}[theorem]{Remark}
\newtheorem{question}[theorem]{Question}

\def\Cset{\mathbb{C}}

 \def\Zset{\mathbb{Z}}

\def\Fset{\mathbb {F}}

\def\va{ \varepsilon}
\def\wt{\widetilde}
\def\wh{\widehat}
\def\leq{\leqslant }
\def\geq{\geqslant}

\def\A{\mathcal{A}}
\def\Ad{\A_d}

\def\R{\mathcal{R}}
\def\Rd{\R_d}
\def\D{\mathcal{D}}
\def\Dd{\D_d}
\def\Gd{\mathfrak G_d}
\def\Zad{\Zset^{\Ad}}
\def\G{\mathfrak G}
\def\H{\mathfrak H}
\def\K{\mathfrak K}
\def\S{\mathfrak S}
\def\MP{{\rm Mod}(\pi)}

\begin{document}

\title[Zorich conjecture for hyperelliptic Rauzy--Veech groups]{Zorich conjecture for hyperelliptic Rauzy--Veech groups}

\author[A. Avila, C. Matheus and J.-C. Yoccoz]{Artur Avila, Carlos Matheus and Jean-Christophe Yoccoz}

\address{Artur Avila:
CNRS UMR 7586, Institut de Math\'ematiques de Jussieu - Paris Rive Gauche,
B\^atiment Sophie Germain, Case 7012, 75205 Paris Cedex 13, France
\&
IMPA, Estrada Dona Castorina 110, 22460-320, Rio de Janeiro, Brazil
}

\email{artur@math.jussieu.fr}

\address{Carlos Matheus:
Universit\'e Paris 13, Sorbonne Paris Cit\'e, CNRS (UMR 7539),
F-93430, Villetaneuse, France.
}

\email{matheus.cmss@gmail.com}

\address{Jean-Christophe Yoccoz:
Coll\`ege de France (PSL), 3, Rue d'Ulm, 75005 Paris, France
}

\email{jean-c.yoccoz@college-de-france.fr}

\date{\today}

\begin{abstract}
We describe the structure of hyperelliptic Rauzy diagrams and hyperelliptic Rauzy--Veech groups. In particular, this provides a solution of the hyperelliptic cases of a conjecture of Zorich on the Zariski closure of Rauzy--Veech groups.
\end{abstract}
\maketitle
%\pageheight{8.5truein}
%\pagewidth{6.5truein}

%\tableofcontents

\section{Introduction}

The Kontsevich--Zorich conjecture provides a precise description of the deviations of ergodic averages of almost every interval exchange transformations and translation flows in terms of the Lyapunov exponents of the Kontsevich--Zorich (KZ) cocycle with respect to the Masur--Veech measures on the strata of moduli spaces of translation surfaces.

After an important partial progress of Forni~\cite{Forni} in 2001, the Kontsevich--Zorich conjecture was fully established by Avila and Viana \cite{AV} in 2007 via the study of certain combinatorial models for the Kontsevich--Zorich cocycles called Rauzy--Veech groups. In a nutshell, Avila and Viana confirmed the Kontsevich--Zorich conjecture by showing that the Rauzy--Veech groups are pinching and twisting.

Nevertheless, Avila and Viana pointed out in \cite[Remark 6.12]{AV} that their methods leave open an  interesting conjecture of Zorich (cf. \cite[Appendix A.3]{Zorich4}) concerning the Zariski denseness of Rauzy--Veech groups in symplectic groups. Indeed, it is known\footnote{See the Appendix \ref{a.pinching+twisting-Zariski} below for a concrete example.} among experts that some pinching and twisting groups have small Zariski closures, so that it is not possible to abstractly deduce\footnote{On the other hand, Zorich conjecture implies Avila--Viana theorem on the pinching and twisting properties for Rauzy--Veech groups. In fact, Zariski density implies the pinching property by the work of Benoist \cite{Benoist}, while the twisting property is automatic (because it has to do with minors of matrices). Hence, our proofs of Theorem \ref{t.AMY-intro} give new proofs of Avila--Viana theorem in the particular case of hyperelliptic Rauzy--Veech groups.} Zorich's conjecture from Avila--Viana techniques.

In this paper, we confirm Zorich conjecture for \emph{hyperelliptic} Rauzy--Veech groups by proving the following stronger result.

\begin{theorem}\label{t.AMY-intro} The Rauzy--Veech group associated to a hyperelliptic connected component of a stratum of the moduli space of genus $g$ translation surfaces is an explicit, finite-index subgroup of the symplectic group $Sp(2g,\mathbb{Z})$.
\end{theorem}

We refer the reader to Theorem \ref{thm1} below for a precise version of this statement. For now, let us just make some comments on the proof of this result.

Rauzy \cite{Rauzy} discovered a particularly beautiful combinatorial description for hyperelliptic Rauzy diagrams. This description allows us to compute the generators of hyperelliptic Rauzy--Veech groups and, more importantly, to relate distinct hyperelliptic Rauzy--Veech groups via an inductive procedure. In particular, we are able to prove Theorem \ref{t.AMY-intro} by induction (on the complexity of the hyperelliptic Rauzy diagrams): see Section \ref{s.AMYhyp1} below.

After we completed the argument in the above paragraph, M\"oller pointed out (in private communication) that our description of hyperelliptic Rauzy--Veech groups shared some similarities with the work \cite{A'Campo} of A'Campo on certain representations of braid groups defined via homological actions on hyperelliptic Riemann surfaces. As it turns out, this is not a coincidence: we show in Section \ref{s.A'Campo} below that the hyperelliptic Rauzy--Veech groups are naturally related to the images of the monodromy representations considered by A'Campo. In particular, the main results of A'Campo's paper \cite{A'Campo} can be used to give another proof of Theorem \ref{t.AMY-intro}.

\begin{remark} This second proof of Theorem \ref{t.AMY-intro} described in the previous paragraph provides more information about hyperelliptic Rauzy diagrams: for instance, we will show that the image of the natural homomorphism from the fundamental group of hyperelliptic Rauzy diagrams to the mapping class group is an infinite-index subgroup called symmetric mapping class group. In particular, the analog of Theorem \ref{t.AMY-intro} at the fundamental group level is not true. See Section \ref{s.A'Campo} for more details.
\end{remark}

The organization of this paper is the following. In Section \ref{s.hypRV}, we recall some basic facts about hyperelliptic Rauzy diagrams and Rauzy--Veech groups, and we state in Theorem \ref{thm1} the precise version of Theorem \ref{t.AMY-intro}. In Section \ref{s.AMYhyp1}, we give our first proof of Theorem \ref{thm1} by induction on the complexity of hyperelliptic Rauzy diagrams. In Section \ref{s.A'Campo}, we give a second proof of Theorem \ref{thm1} based on the interpretation of hyperelliptic Rauzy--Veech groups in terms of certain monodromy representations of braid groups. In particular, Sections \ref{s.AMYhyp1} and \ref{s.A'Campo} can be read independently of each other. Finally, we exhibit in Appendix \ref{a.pinching+twisting-Zariski} an example of pinching and twisting group with small Zariski closure in order to justify our assertion that Zorich conjecture can not be abstractly reduced to the results of Avila--Viana \cite{AV}.

\begin{remark} In a forthcoming paper \cite{AMY}, we will use the framework of this article to analyze the Kontsevich-Zorich cocycle over certain loci of cyclic covers of hyperelliptic connected components of strata of the moduli space of translation surfaces.
\end{remark}

\begin{remark} In a recent preprint \cite{EFW}, Eskin, Filip and Wright studied the algebraic hull of the Kontsevich--Zorich cocycle and the monodromies associated to general ergodic $SL(2,\mathbb{R})$-invariant probability measures on moduli spaces of translation surfaces. The notion of Rauzy--Veech groups shares some similarities with the algebraic hulls and the monodromies of Masur--Veech measures: roughly speaking, Rauzy--Veech groups, resp. algebraic hulls, resp. monodromies, are related to matrices obtained by following certain orbits of the Teichm\"uller flow, resp. orbits of $SL(2,\mathbb{R})$, resp. arbitrary paths in connected components of strata of moduli spaces of translation surfaces. In particular, one has that Rauzy--Veech groups are subgroups of the monodromies of Masur--Veech measures. Consequently, our Theorem \ref{t.AMY-intro} implies that monodromies of hyperelliptic Masur--Veech measures are commensurable to arithmetic lattices of symplectic groups: this refines Corollary 1.7 in Filip's article \cite{Fi} in this particular setting. On the other hand, the relation between Rauzy--Veech groups and algebraic hull of Masur--Veech measures is not so obvious (partly because the definition of algebraic hull involves representing matrices in \emph{a priori} unknown measurably chosen bases) and, thus, it is not clear that our Theorem \ref{t.AMY-intro} provides any new information related to Corollary 1.4 in Eskin--Filip--Wright paper \cite{EFW}.
\end{remark}

\subsection*{Acknowledgements} The authors are thankful to Pascal Hubert and Martin M\"oller for pointing out to us the references \cite{Rauzy} and \cite{A'Campo}. Also, the authors are grateful to the two referees for their careful reading of this text.

%%%%%%%%%%%%%%%%%%%%%%%%%%%%%%%%%%%%%%%%%%%%%%%%%%%%%%%
\section{The hyperelliptic Rauzy--Veech group}\label{s.hypRV}
%%%%%%%%%%%%%%%%%%%%%%%%%%%%%%%%%%%%%%%%%%%

In this entire section, we will assume that the reader has some familiarity with the lecture notes \cite{Y-Pisa} by the third author of this paper. Also, let us point out that the facts stated in the next subsection are just reformulations (in our notations) of the results obtained by Rauzy \cite[Section 4]{Rauzy}.

\subsection {Hyperelliptic Rauzy diagrams: notations and description}\label{ss.hypRauzy} Let $d \geq 2$ be an integer. Let $\Ad$ be the alphabet whose $d$ elements are the integers in arithmetic progression $d-1, \, d-3, \ldots , 1-d$. Let $\iota$ be the involution $k \mapsto -k$ of $\Ad$. We define inductively the hyperelliptic Rauzy class $\Rd$ over $\Ad$ and the associated hyperelliptic Rauzy diagram $\Dd$.
The Rauzy class $\Rd$ contains a central vertex $\pi^* = \pi^*(d) = (\pi^*_t(d),\pi^*_b(d))$ associated to the pair of bijections
$\pi^*_t(d):\Ad\to\{1,\dots,d\}$ and $\pi^*_b(d):\Ad\to\{1,\dots,d\}$ defined by 
$$\pi^*_t(d)(k) = \frac 12 (d+1+k),\quad \pi^*_b(d)(k) = \frac 12 (d+1-k).$$

For $d=2$, this is the only vertex. For $d \geq 2$, $\R_{d+1}$ is the disjoint union of $\pi^*(d+1)$, $j_t(\Rd)$ and $j_b(\Rd)$, where the injective maps $j_t$, $j_b$ are defined as follows: for  $\pi \in \Rd$, writing $j_t(\pi) = t\pi$, $j_b(\pi) = b\pi$, we have that $t\pi=(t\pi_t,t\pi_b)$ and $b\pi=(b\pi_t,b\pi_b)$ are given by the bijections from $\Ad$ to
$\{1,\dots,d\}$ described by the formulas 
$$t\pi_t(-d) =1,\quad \quad \quad t\pi_b(-d) = \pi_b(d-3) ,$$
$$t\pi_t(k) = 1+\pi_t(k-1),$$
$$  t\pi_b(k) = \left \{
\begin{array}{cc}
 \pi_b(k-1)& \text{if } \pi_b(k-1) < \pi_b(d-3),\\
\pi_b(k-1)+1 & \text{if } \pi_b(k-1) \geq \pi_b(d-3),
\end{array} \right.$$
for $2-d \leq k \leq d$, and
$$b\pi_b(d) =1,\quad \quad \quad b\pi_t(d) = \pi_t(3-d) ,$$
$$b\pi_b(k) = 1+\pi_b(k+1),$$
$$  b\pi_t(k) = \left \{
\begin{array}{cc}
 \pi_t(k+1)& \text{if } \pi_t(k+1) < \pi_t(3-d),\\
\pi_t(k+1)+1 & \text{if } \pi_b(k+1) \geq \pi_t(3-d),
\end{array} \right.$$
for $-d \leq k \leq d-2$.

\begin{figure}[h!]
%\centering
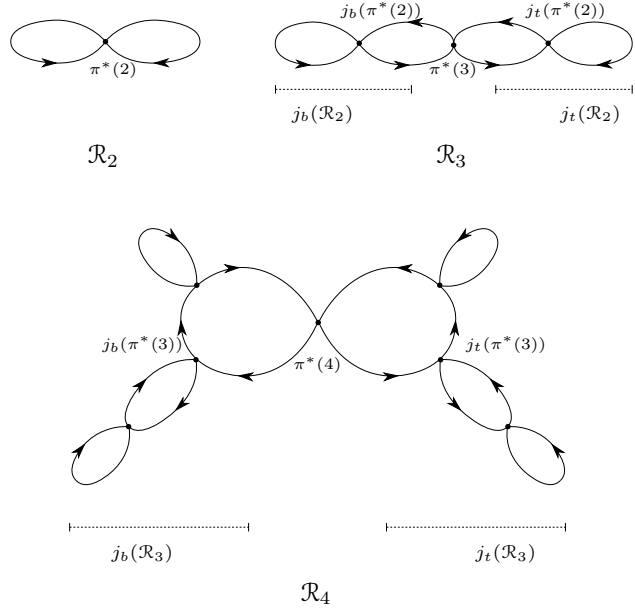\caption{Geometry of the hyperelliptic Rauzy classes $\mathcal{R}_2$, $\mathcal{R}_3$ and $\mathcal{R}_4$.}
\end{figure}

The one-to-one maps $R_t$, resp. $R_b$ from $\Rd$ to itself determining the arrows of $\Dd$ of top, resp. bottom type verify
$$\left \{  \begin{array}{cc} R_t(\pi^*(d+1)) = j_t (\pi^*(d)),\\
 R_b(\pi^*(d+1)) = j_b (\pi^*(d)), \end{array} \right. \quad \left \{  \begin{array}{cc} R_t \circ j_b \circ R_t^{-1} =j_b, \\
R_b \circ j_t \circ R_b^{-1} =j_t, \end{array} \right. $$

$$ \left \{  \begin{array}{cc} R_t \circ j_t \circ R_t^{-1}(\pi) =j_t(\pi), \quad \quad \pi \ne \pi^*(d),\\
R_b \circ j_b \circ R_b^{-1}(\pi) =j_b(\pi), \quad \quad \pi \ne \pi^*(d),\end{array} \right.$$

$$R_t \circ j_t \circ R_t^{-1}(\pi^*(d)) = \pi^*(d+1)= R_b \circ j_b \circ R_b^{-1}(\pi^*(d)).$$

The involution $I_d$ on  $\Rd$ defined by $I_d((\pi_t, \pi_b)) := (\pi_b \circ \iota,\pi_t \circ \iota)$ satisfies
$$I_d(\pi^*(d)) = \pi^*(d), \quad I_{d+1}\circ j_b \circ I_d = j_t , \quad I_d \circ R_b \circ I_d =  R_t.$$

There is a natural one-to-one correspondence $W_d$ between the elements of $\Rd$ and the words in $\{t,b\}$ of length $<d-1$: namely, $W_d(\pi^*(d))$ is the empty word, $W_d (j_t(\pi))$ is the word $tW_{d-1}(\pi)$ and $W_d (j_b(\pi))$ is the word $bW_{d-1}(\pi)$. The involution $I_d$ corresponds to the exchange of the letters $t,b$. One has also
$$W_d (R_t(\pi)) = W_d(\pi)t, \quad W_d (R_b(\pi)) = W_d(\pi)b, \quad \text {if } |W_d(\pi)|<d-2.$$

When $|W_d(\pi)|= d-2$, one writes $W_d(\pi)= W't^m $ with $m\geq 0$ and $W'$ empty or finishing by $b$; one has then $W_d (R_t(\pi))=W'$.  Similarly for $W_d (R_b(\pi))$.

It is also not difficult to recover from $W_d(\pi)$ the winners of the arrows starting from $\pi$: the winner of the arrow of top type starting from $\pi$ is the letter $d-1-2w_b(\pi)$ of $\Ad$, where $w_b(\pi)$ is the number of occurrences of $b$ in $W_d(\pi)$; similarly, the winner of the arrow of bottom type starting from $\pi$ is the letter $1-d+2w_t(\pi)$ of $\Ad$. Observe that we have always
$$ d-1-2w_b(\pi)> 1-d+2w_t(\pi).$$

Another useful property of the hyperelliptic Rauzy diagrams is the following: given any vertex $\pi \in \Rd$, there is an \emph{unique} oriented \emph{simple}\footnote{A path is {\it simple} if it does not pass more than once through any vertex.} path in $\Dd$ from $\pi^*(d)$ to $\pi$. Indeed, this is best seen via the correspondence $W_d$ above: the length of such a path is $|W_d(\pi)|$ and the path itself is through the sequence of initial subwords of $W_d(\pi)$. We will denote by $\gamma^*(\pi)$ this path.

Observe that all simple loops of positive length in $\Rd$ are {\it elementary}, that is, they are made of arrows of the same type (and consequently with the same winner). For any such loop $\gamma$, there is a unique vertex $\pi$ such that $\gamma$ passes through $\pi$ but $\gamma^*(\pi)$ does not contain any arrow of $\gamma$. As it turns out, $\pi$ is the vertex of $\gamma$ such that $ |W_d(\pi)|$ is minimal. One has
$$|\gamma| + |W_d(\pi)| = d-1.$$

\subsection{The hyperelliptic Rauzy--Veech group}

Let $\gamma$ be an elementary simple loop in $\Rd$ and denote by $\pi$ the vertex of $\gamma$ with
$|W_d(\pi)|$ minimal (see above).
Let $\gamma'$ be the non-oriented loop based at $\pi^*(d)$ defined by $\gamma' = \gamma^*(\pi) * \gamma * (\gamma^*(\pi))^{-1}$.

\subsubsection{Calculation of some Kontsevich--Zorich matrices} Let us compute the matrix $B_{\gamma'}$ associated to $\gamma'$ by the Rauzy--Veech algorithm / Kontsevich--Zorich cocycle (see Subsection 7.5 in \cite{Y-Pisa} for definitions). Assume for instance that the loop $\gamma$ is of top type. We have $w_b(\pi) + w_t(\pi) + |\gamma| = d-1$.

On one hand, the winner of all the arrows of $\gamma$ is $d-1-2w_b(\pi)$. On the other hand, starting from $\pi$, the losers are successively  $1-d + 2w_t(\pi), 3-d +2w_t(\pi) , \ldots, d-3-2w_b(\pi)$. Therefore, by   writing $B_{\gamma}(v) = v'$, we have
$$v'_k= \left \{  \begin{array}{cc}
v_k + v_{d-1-2w_b(\pi)} & \text{if } 1-d + 2w_t(\pi) \leq k < d-1-2w_b(\pi),\\
v_k & \text{otherwise}
\end{array} \right.$$

If $W_d(\pi)$ is empty, i.e., $w_b(\pi)= w_t(\pi) =0$, we have $\gamma'=\gamma$ and $B_{\gamma'}=B_{\gamma}$. So, we can assume now that $W_d(\pi)$ is not empty. Let $w_1$ be the number of occurrences of $b$ at the end of $W_d(\pi)$; one has $0 < w_1 \leq w_b(\pi)$. Write $\gamma^*(\pi) = \gamma^1 * \gamma_1$ with $|\gamma_1| = w_1$.

The winner of all the arrows of $\gamma_1$ is $1-d+2w_t(\pi)$, while the losers are successively $d-1-2w_b(\pi)+2w_1,  \ldots, d-1-2w_b(\pi)+2w_1$. Thus, by writing $B_{\gamma_1}(v) = \wh v$, we have
$$\wh v_k= \left \{  \begin{array}{cc}
v_k + v_{1-d+2w_t(\pi)} & \text{if } d-1-2w_b(\pi) < k \leq  d-1-2w_b(\pi)+2w_1 ,\\
v_k & \text{otherwise}
\end{array} \right.$$

For $ B_{\gamma_1 * \gamma * \gamma_1^{-1}}(v)=v'$, we have therefore
$$v'_k= \left \{  \begin{array}{cc}
v_k + v_{d-1-2w_b(\pi)} & \text{if } 1-d + 2w_t(\pi) \leq k < d-1-2w_b(\pi),\\
v_k - v_{d-1-2w_b(\pi)} & \text{if } d-1-2w_b(\pi) < k \leq  d-1-2w_b(\pi)+2w_1 ,\\
v_k & \text{otherwise}
\end{array} \right.$$

Let $\gamma'_1:= \gamma_1 * \gamma * \gamma_1^{-1}$. If $\gamma^1$ is empty, i.e., $w_b(\pi)=w_1$, $w_t(\pi)=0$, we have $\gamma'= \gamma'_1$ and the computation of $B_{\gamma'}$ is complete. Otherwise, we go on by writing $\gamma^1= \gamma^2 * \gamma_2$, with $\gamma_2$ made of the arrows of top type ending $\gamma^1$. By writing $|\gamma_2| = w_2>0$ and $B_{\gamma_2}(v) = \wh v$, one has
$$\wh v_k= \left \{  \begin{array}{cc}
v_k + v_{d-1-2w_b(\pi)+2w_1} & \text{if }  1-d+2w_t(\pi) -2 w_2\leq k <  1-d+2w_t(\pi) ,\\
v_k & \text{otherwise}
\end{array} \right.$$

Let $\gamma'_2= \gamma_2 * \gamma'_1 * \gamma_2^{-1}$. For $B_{\gamma'_2}(v)=v'$, we have
$$v'_k= \left \{  \begin{array}{cc}
v_k + v_{d-1-2w_b(\pi)} & \text{if } 1-d + 2w_t(\pi)-2w_2 \leq k < d-1-2w_b(\pi),\\
v_k - v_{d-1-2w_b(\pi)} & \text{if } d-1-2w_b(\pi) < k \leq  d-1-2w_b(\pi)+2w_1 ,\\
v_k & \text{otherwise}
\end{array} \right.$$

We go on till $\gamma^m$ is empty. In this way, for $B_{\gamma'} (v)=v'$, one obtains:
$$v'_k= \left \{  \begin{array}{cc}
v_k + v_{d-1-2w_b(\pi)} & \text{if }  k < d-1-2w_b(\pi),\\
v_k - v_{d-1-2w_b(\pi)} & \text{if }  k >  d-1-2w_b(\pi),\\
v_k & \text{if } k =  d-1-2w_b(\pi).
\end{array} \right.$$
Note that this formula depends only on the type and the winner of $\gamma$.

Similarly, when $\gamma$ has bottom type, the formula for $B_{\gamma'} (v)=v'$ is
$$v'_k= \left \{  \begin{array}{cc}
v_k + v_{1-d+2w_t(\pi)} & \text{if }  k >1-d+2w_t(\pi),\\
v_k - v_{1-d+2w_t(\pi)} & \text{if }  k < 1-d+2w_t(\pi),\\
v_k & \text{if } k =  1-d+2w_t(\pi).
\end{array} \right.$$

\begin{remark}\label{r.monoid-group}
The matrix $B_{\gamma'}$ associated to a loop $\gamma$ of bottom type is the inverse of the matrix corresponding to the loop of top type and the same winner as $\gamma$.
\end{remark}

\begin{remark}\label{r.homology-homotopy} Actually, one can completely describe the action of $\gamma'$ at the homotopy level (instead of the homology level of the matrix $B_{\gamma'}$) and, again, it depends only on the type and winner of $\gamma$. We will come back to this point later in Section \ref{s.A'Campo} below.
\end{remark}

\subsubsection{Definition of the hyperelliptic Rauzy--Veech groups} Let $\gamma$ be the elementary simple loop in $\Rd$ of bottom type with winner $p\in\Ad$. Our previous discussion shows that the matrix $B_{\gamma'}$ corresponds to the operator $L_p $ on  $\Zset^{\Ad}$ given by
$$L_p(e_q) = \left \{ \begin{array}{cc}
e_q    & \text{if} \; q\ne p  \\
-\sum_{r<p} e_r + \sum_{r\geq p} e_r   &  \text{if} \; q=p
  \end{array} \right.$$
where $(e_p)$ is the canonical basis of $\Cset^{\Ad}$. Also, by Remark \ref{r.monoid-group}, the matrix  associated to the elementary simple loop in $\Rd$ of top type with winner $p\in\Ad$ corresponds to the inverse of the operator $L_p$.

\begin{definition}\label{def1}
The {\it hyperelliptic Rauzy--Veech group of complexity $d$} is the subgroup $\mathfrak G_d$ of $SL(\Zset^{\Ad})$ generated by the operators $L_p$, $p \in \Ad$.
\end{definition}

\subsubsection{Intersection form}
%%%%%%%%%%%%%%%%%%%%%%%%%%%%%%%%%%%%%%%%%
The antisymmetric matrix $\Omega = \Omega(d)$ with entries
$$ \Omega_{pq} := \left\{ \begin{array}{ccc}+1&\text{if}&p<q\\ -1&\text{if}&p>q\\ 0&\text{if}& p=q \end{array} \right.$$
indexed by $\Ad \times \Ad$ can be interpreted as the intersection form on the homology of certain translation surfaces (see Subsections 3.4 and 4.5 of \cite{Y-Pisa}).

The operators $L_p$ satisfy $ L_p \ \Omega \ ^t \negthinspace L_p = \Omega$ and, \emph{a fortiori},  the same is true of all the elements of $\Gd$:
\begin{equation}\label{e.symplecticity}
B \ \Omega \ ^t \negthinspace B = \Omega, \qquad \forall B \in \Gd.
\end{equation}

%%%%%%%%%%%%%%%%%%%%%%%%%%%%%%%%%%%%%%%%%
\subsubsection{Symplecticity for $d$ even}
%%%%%%%%%%%%%%%%%%%%%%%%%%%%%%%%%%%%%%%%
The matrix $\Omega (d)$ corresponds to the intersection form on the absolute homology of certain translation surfaces when $d$ is even. In particular, $\Omega(d)$ is unimodular.
The symplectic form on $\mathbb{Z}^{\Ad}$ is defined by
 $$ (v,v') \mapsto ^ t \negthinspace v \  \Omega^{-1} v'.$$
The relation \eqref{e.symplecticity} shows that $\Gd \subset Sp(\Omega^{-1}(d), \Zset)$. Note that the group $Sp(\Omega^{-1}(d), \Zset)$ is isomorphic to $Sp(d,\Zset)$.

%%%%%%%%%%%%%%%%%%%%%%%%%%%%%%%%%%%%%%%%%%
\subsubsection{The case of $d$ odd}
%%%%%%%%%%%%%%%%%%%%%%%%%%%%%%%%%%%%%%%%
Assume that $d$ is odd.
\begin{lemma}\label{lem1}
The matrix $\Omega(d)$ has rank $d-1$. The image $\Omega(d)(\Zset^{\A_d})$ is the hyperplane
$$ H(d) = \left\{ v \in \Zad \; \Big\vert \; \sum_{p \in \Ad} (-1)^{p/2} v_p =0 \right\}.$$
The kernel of  $\Omega(d)$ is generated by $h^{\star} := \sum_{p\in \A_d}  (-1)^{p/2} e_p$
\end{lemma}
\begin{proof}
The vectors $\Omega (d) e_p$ belong to $H(d)$. Moreover, one has
$$ \Omega(d) (e_{p+2} - e_p) = e_p + e_{p+2}, \qquad \forall \, p \in \A_d, \, p < d-1,$$
and $\{e_p + e_{p+2} : p \in \A_d, \, p < d-1\}$ form a basis of $H(d)$. Finally, it is clear that $\Omega(d).h^{\star} =0$.
\end{proof}

\begin{proposition}\label{prop1}
The matrices in $\G_d$ satisfy $ ^t h^{\star} B = ^t \negmedspace h^{\star}$
\end{proposition}
\begin{proof}
Indeed this is the case for each $L_p$, $p\in \A_d$.
\end{proof}

The symplectic form induced by $\Omega(d)$ on $H(d)$ is defined as follows: for $v,v' \in H(d)$,  $v = \Omega w, v' = \Omega w'$, we set
$$ (v,v') \mapsto ^t \negthinspace w\ \Omega\ w' .$$

Observe that this does not depend on the choices of $w,w'$.

\begin{remark}\label{rem1}
This is coherent with the definition of the symplectic form for $d$ even.
\end{remark}

From \eqref{e.symplecticity} (or Proposition \ref{prop1}), the elements of $\Gd$ preserve the hyperplane $H(d)$ and their restrictions to $H(d)$ are symplectic with respect to the symplectic form on $H(d)$.

We denote still by $Sp(\Omega^{-1}(d),\Zset)$ the group of operators in $SL(\Zad)$ satisfying \eqref{e.symplecticity} (although $\Omega$ is not invertible in this case).

%%%%%%%%%%%%%%%%%%%%%%%%%%%%%%%%%%%%%%%
\subsubsection{Reduction modulo $2$}
%%%%%%%%%%%%%%%%%%%%%%%%%%%%%%%%%%%%%%%%

For $p \in \Ad$, let $\bar L_p$ be the reduction mod.$2$ of $L_p$: it acts on $(\Zset/2)^{\Ad}$. Denote by $(\bar e_p)_{p\in \Ad}$ the canonical basis of $(\Zset/2)^{\Ad}$ and define $\bar e^{\star} := \sum_{p\in \Ad} \bar e_p$.

\begin{proposition}\label{prop2}
For any $q \in \A_d$, the $d+1$ vectors $e^{\star}$, $e_p$, $p \in \Ad$, are permuted by $\bar L_q$. More precisely, $\bar L_q$ fixes $\bar e_p$ for $p\ne q$ and exchanges $\bar e_q$ and $\bar e^{\star}$.
\end{proposition}

\begin{proof}
This follows easily from the definitions.
\end{proof}

\begin{corollary}\label{cor1}
The image of $\Gd$ in $SL(\Fset_2^{\A_d})$ is the subgroup $\H_d$ formed of elements preserving
$\mathcal E := \{\bar e^{\star}\}\cup\{\bar e_p\}_{p \in \Ad}$. It is isomorphic to the symmetric group of order $d+1$.
\end{corollary}

\begin{proof}
Indeed, this is a direct consequence of Proposition \ref{prop2} because the group generated by the transpositions $(0,i)$, $1 \leq i \leq d$ is the full symmetric group of $\{0,\ldots,d\}$.
\end{proof}

%%%%%%%%%%%%%%%%%%%%%%%%%%%%%%%%%%%%%%
\subsection{Statement of the main result}\label{ssStat}
%%%%%%%%%%%%%%%%%%%%%%%%%%%%%%%%%%%%%%%

The following statement provides a precise version for Theorem \ref{t.AMY-intro} above.
\begin{theorem}\label{thm1}
For any integer $d \geq 2$, the group $\G_d$ consists of matrices $B \in Sp(\Omega^{-1}(d),\Zset)$ whose image in $SL(\mathbb{F}_2^{\A_d})$ belongs to $\H_d$.
\end{theorem}
Here, we recall that the group $Sp(\Omega^{-1}(d),\Zset)$ was defined using a special convention when $d$ is odd (see the two paragraphs after Remark \ref{rem1} above). 

In the sequel, we will give two proofs of this result in Sections \ref{s.AMYhyp1} and \ref{s.A'Campo}.

More precisely, our discussion in Section \ref{s.AMYhyp1} below will establish (by induction) this theorem at the same time of the next two results.

\begin{theorem}\label{thm2}
For any even integer $d \geq 2$ and  any $p \in \A_d$, the orbit of $e_p$ under $\G_d$ is equal to the set of primitive vectors in $\Zset^{\A_d}$ which are congruent mod.$2$ to a vector in the set $\mathcal E$ from  Corollary \ref{cor1}.
\end{theorem}

\begin{theorem}\label{thm3}
For any odd integer $d \geq 3$, any $p \in \A_d$, the orbit of $e_p$ under $\G_d$ is equal to the set of primitive vectors in $\Zset^{\A_d}$ which are congruent mod.$2$ to a vector in  $\mathcal E$ and belong to the affine hyperplane
$$ \{v \in \Zset^{\A_d} \vert ^t h^{\star}.(v-e_p) =0\} .$$
\end{theorem}

On the other hand, our discussion in Section \ref{s.A'Campo} below will establish Theorem \ref{thm1} by expanding on Remark \ref{r.homology-homotopy} above, that is, we will use the relationship between hyperelliptic Rauzy--Veech groups and certain monodromy representations of braid groups in order to reduce Theorem \ref{thm1} to some results of A'Campo \cite{A'Campo}.

\section{Proof of Theorem \ref{thm1}}\label{s.AMYhyp1}

In this section, we prove Theorems \ref{thm1}, \ref{thm2} and \ref{thm3} by induction on the integer $d \geq 2$. In the initial case $d=2$, it is well known that the group generated by $L_{-1}$ and $L_1$ is equal to $SL(\Zset^{\A_2})$.
Observe that $\H_2$ is equal to $SL(\Fset_2^{\A_2})$ and $Sp(\Omega^{-1}(2),\Zset)$ is equal to $SL(\Zset^{\A_2})$. Therefore Theorem \ref{thm1} holds for $d=2$.
Any primitive vector in $\Zset^{\A_2}$ belongs to the orbit of $e_1$ (or $e_{-1}$) under
$SL(\Zset^{\A_2})$. Therefore Theorem \ref{thm2} also holds for $d=2$.

In the sequel, we denote by $\G'_d$ the group of matrices $B \in Sp(\Omega^{-1}(d),\Zset)$ whose image in $SL(\Fset_2^{\A_d})$ belongs to $\H_d$. By Proposition \ref{prop2} and relation \eqref{e.symplecticity}, the group $\G_d$ is contained in $\G'_d$. In this setting, our task of showing Theorem \ref{thm1} consists in proving that $\G_d$ is equal to $\G'_d$.

%%%%%%%%%%%%%%%%%%%%%%%%%%%%%%%%%%%%%%%
\subsection{Stabilizer of $e_{1-d}$ in $\G'_d$}
%%%%%%%%%%%%%%%%%%%%%%%%%%%%%%%%%%%%%%%%

A matrix $B$ belonging to the stabilizer $\K_d$ of $e_{1-d}$ in $\G'_d$ can be written in the block form
\begin{equation}\label{eq1}
B:= \left(\begin{array}{cc} 1&v\\ 0&g\end{array} \right).
\end{equation}

Here, $v$ is a  integral line vector of dimension $d-1$ and $g$ is an unimodular square matrix of dimension $d-1$. Both are indexed by $\A_d \setminus \{1-d\}$, which is equal to $\A_{d-1}$ shifted by $1$. When considering the stabilizer $\K_d$, we will forget the shift and think of $v,g$ as indexed by $\A_{d-1}$.

Let $e^{\sharp}:= \sum_{p \in \A_d\setminus \{1-d\}} e_p$. By writing
$$ \Omega(d) = \left( \begin{array}{cc} 0 & ^t e^{\sharp} \\ -e^{\sharp} &\Omega(d-1) \end{array} \right),$$
the relation \eqref{e.symplecticity} is equivalent to
\begin{equation}\label{eq2}
\left\{ \begin{array} {rcl} g\ \Omega(d-1)\ ^t \negthinspace g &=&\Omega(d-1) \\ \Omega(d-1)\ ^t  v &=& e^{\sharp} - g^{-1} e^{\sharp} \end{array} \right.
\end{equation}

Here, the first relation means that $g \in Sp(\Omega^{-1}(d-1),\Zset)$. The map $B \mapsto g$ defines a homomorphism $\varphi_d$ from $\K_d$ to $Sp(\Omega^{-1}(d-1),\Zset)$.

\begin{proposition}\label{prop3}
The image of this homomorphism is equal to $\G'_{d-1}$.
\end{proposition}
\begin{proof}
First, the image is contained into $\G'_{d-1}$: if $B$ is congruent mod.$2$ to a matrix in $\H_d$, $g$ is congruent mod.$2$ to a matrix in $\H_{d-1}$.
\smallskip
For the converse, let $g \in \G'_{d-1}$. We first observe that, when $d$ is even, the vector $e^{\sharp} - g^{-1} e^{\sharp}$ is contained in the image $H(d-1)$ of $\Omega(d-1)$: indeed, one has (with $h^{\star} = \sum_{p \in \A_{d-1}} (-1)^{p/2} e_p$)
\begin{equation}\label{eq3}
^t h^{\star} (g e -e)  =0, \qquad \forall \, g \in Sp(\Omega^{-1}(d-1),\Zset), \ \forall \, e \in \Zset^{\A_{d-1}}
\end{equation}
according to Proposition \ref{prop1}.

We now check that it is always possible to choose a solution $v$ of the second equation in (\ref{eq2}) such that $B$ is congruent mod.$2$ to a matrix in $\H_d$.
There are two cases:
\begin{itemize}
\item The reduction mod.$2$ of $g$ permutes the $\bar e_p$, $p \in \A_{d-1}$.
In this case, the vector $e^{\sharp} - g^{-1} e^{\sharp}$ is even and one can find an even vector $v$ which satisfies the second equation of (\ref{eq2}). Then the reduction mod.$2$ of $B$ belongs to $\H_d$.
\item There exists $p \in \A_{d-1}$ such that $g.e_p$ is congruent mod.$2$ to $e^{\sharp}$. Then $e^{\sharp} - g^{-1} e^{\sharp}$ is congruent mod.$2$ to $e^{\sharp} - e_p$, which is itself congruent mod.$2$ to $\Omega(d-1) e_p$. Therefore one can find a solution $v$ of the second equation of (\ref{eq2}) which is congruent mod.$2$ to ${}^t e_p$. Then the reduction mod.$2$ of $B$ belongs to $\H_d$.
\end{itemize}

This proves the proposition.
\end{proof}

\begin{proposition}\label{prop4}
When $d$ is odd, the homomorphism $\varphi_d$ is an isomorphism.
\end{proposition}
\begin{proof}
Indeed, $\Omega(d-1)$ is invertible in this case, hence the second equation in \eqref{eq2} has an unique solution.
\end{proof}

\begin{proposition}\label{prop5}
When $d$ is even, two matrices $B_0,B_1 \in \K_d$ as in (\ref{eq1}) have the same image under $\varphi_d$ if and only if the difference $v_1 -v_0$ of the corresponding vectors is an even multiple of $^t h^{\star}$.
\end{proposition}
\begin{proof}
Indeed $h^{\star}$ is a basis of the $1$-dimensional kernel of $\Omega(d-1)$. The assertion of the proposition results from the end of the proof of Proposition \ref{prop3}.
\end{proof}

The stabilizer $\K_d$ of $e_{1-d}$ in $\G'_d$ is completely described by Propositions \ref{prop3}, \ref{prop4}, \ref{prop5}.

%%%%%%%%%%%%%%%%%%%%%%%%%%%%%%%%%%%%%%%
\subsection{The subgroups $\S_{p,q}$ of $\G_d$}
%%%%%%%%%%%%%%%%%%%%%%%%%%%%%%%%%%%%%%%

Let $p < q$ be distinct elements of $\A_d$. We denote by $\S_{p,q}$ the subgroup of $\G_d$ generated by $L_p$ and $L_q$.

Let $B \in \S_{p,q}$. As the vectors $e_r$, $r \ne p,q$, are fixed by both $L_p$ and $L_q$, we have $B.e_r = e_r$ for such $r$. We denote by $B^{\sharp}$ the $2 \times 2$ matrix $$ B^{\sharp}  := \left( \begin{array}{cc} B_{p,p} & B_{p,q} \\ B_{q,p} & B_{q,q} \end{array} \right).$$

\begin{lemma}\label{lem2}
\begin{enumerate}
\item The map $B \mapsto B^{\sharp}$ is an isomorphism from $\S_{p,q}$ onto $SL(2,\Zset)$.
\item The other coefficients of $B$ in the $p$-th and $q$-th columns are given by

$$ B_{r,p} =  \left\{ \begin{array}{rcc} -1+B_{p,p} - B_{q,p} & {\rm if } & r<p \\ -1+B_{p,p} + B_{q,p} & {\rm if } & p<r<q \\ 1 - B_{p,p} + B_{q,p} & {\rm if } & r>q \end{array} \right.$$

$$ B_{r,q} =  \left\{ \begin{array}{rcc} 1+B_{p,q} - B_{q,q} & {\rm if } & r<p \\ -1+B_{p,q} + B_{q,q} & {\rm if } & p<r<q \\ -1 - B_{p,q} + B_{q,q} & {\rm if } & r>q \end{array} \right.$$
\end{enumerate}
\end{lemma}

\begin{proof}
The case $d=2$ (cf. the end of Subsection \ref{ssStat}) shows that $\psi: B \mapsto B^{\sharp}$ is onto $SL(2,\Zset)$. Let $W$ be a word in $L_p^{\pm 1}, L_q^{\pm1}$, $B$ the corresponding element of $\S_{p,q}$. We show, by induction on the length of $W$, that $B$ is determined by $\psi(B)$, with the formulas of the lemma. This is true when $W$ is the empty word. When $W$ has positive length, let $w$ be the last letter of $W$, write $W = W'.w$ and let $B'$ be the matrix associated to $W'$. If for instance $w = L_q^{\eta}$, $\eta \in \{\pm 1\}$, one has, for all $r \in \A_d$
$$ \left\{ \begin{array}{ccl} B_{r,p} & = & B'_{r,p} \\ B_{r,q} & = & B'_{r,q} - \eta B'_{r,p} + \eta \va\end{array} \right. \quad \textrm{where} \quad \va := \left\{ \begin{array}{rcc} -1 & {\rm if } & r<p \\
                                                          0 & {\rm if } & r=p \\
                                                           -1 & {\rm if } & p<r<q \\
                                                            0 & {\rm if } & r=q \\
                                                             1 & {\rm if } & r>q \\ \end{array} \right.$$
It follows that the formulas of the lemma for the $B'_{r,p}, B'_{r,q}$ imply the same formulas for the $B_{r,p}, B_{r,q}$. One deals similarly with the case $w = L_p^{\eta}$. This proves the lemma.
\end{proof}

%%%%%%%%%%%%%%%%%%%%%%%%%%%%%%%%%%%%%%%%%%%%%%%%%%%%%%%%%%%%%%%%%%%%%%%%%%%%%%%%%%
\subsection{The induction step in the odd case}
%%%%%%%%%%%%%%%%%%%%%%%%%%%%%%%%%%%%%%%%%%%%%%%%%%%%%%%%%%%%%%%%%%%%%%%%%%%%%%%%%%%%

In this subsection, we assume that $d \geq 3$ is odd and that Theorems \ref{thm1}, \ref{thm2} hold for $d-1$.

\begin{proposition}\label{prop6}
The stabilizer of $e_{1-d}$ in $\G_d$ is equal to $\K_d$. It is generated by the $L_p$, $p\in \A_d$, $p\ne 1-d$.
\end{proposition}
\begin{proof}
As $\G_d \subset \G'_d$, this stabilizer is contained in $\K_d$. Conversely, let $B \in \K_d$. Write $B$ as in (\ref{eq1}). As $\G_{d-1} = \G'_{d-1}$ by the induction hypothesis, there exists in the subgroup generated by the $L_p$, $p\in \A_d, p\ne 1-d$ a matrix in form (\ref{eq1}) with the same image than $B$ in $\G_{d-1}$. This matrix has to be equal to $B$ by Proposition \ref{prop4}.
\end{proof}

\begin{proposition}\label{prop7}
The orbit of $e_{d-1}$ under $\G_d$ is equal to the set $\mathcal O_d$ of primitive vectors in
$\Zset^{\A_d}$ which belong to the affine hyperplane $\{^t h^{\star}.(v-e_{d-1}) =0\}$ and are congruent mod.$2$ to a vector in $\mathcal E = \{\sum_{p\in \Ad} \bar e_p\}\cup\{\bar e_p\}_{p \in \Ad}$.
\end{proposition}

\begin{proof}
The set  $\mathcal O_d$ contains $e_{d-1}$ and satisfies $L_p(\mathcal O_d) \subset \mathcal O_d$ for all $p \in \A_d$. Hence, $\mathcal{O}_d$ contains the orbit of $e_{d-1}$ under $\G_d$.

Conversely, let $v = \sum_{p\in \A_d} v_p e_p$ be a vector in $\mathcal O_d$.
\begin{itemize}
\item Assume first that the vector $v':= \sum_{p\in \A_d\setminus\{1-d\}} v_p e_p$ is primitive.
In particular, it is not even, hence it is congruent mod.$2$ to either some $e_p$ ($p \in \A_d , p\ne1-d$) or to $ \sum_{p\in \A_d , p\ne1-d}  e_p$. By Theorem \ref{thm2} for $(d-1)$, there exists $g \in \G_{d-1}$ such that $g.e_{d-1} = v'$ (we shift by $1$ the indices and consider $e_{d-1}$ and $v'$ as vectors in $\Zset^{\A_{d-1}}$). Let $B$ be the matrix in $\K_d$ associated to $g$ by (\ref{eq1}) and Proposition \ref{prop4}.
The vector $B.e_{d-1}$ (now in $\Zset^{\A_d}$) is equal to $v$ because $ ^t h^{\star}.(v-B.e_{d-1}) =0$. As $\K_d \subset \G_d$, this proves that $v$ belongs to the orbit of $e_{d-1}$ under $\G_d$.
\item In the general case, Lemma \ref{lem2} says that one can find, in the subgroup generated $L_{1-d}, L_{3-d}$, a matrix $B$ such that
$$ B_{1-d,1-d} v_{1-d} + B_{1-d,3-d} v_{3-d} =0.$$
Then, the $(1-d)$-component of $B.v$ is equal to zero and $B.v$ satisfies the hypothesis of the first case. We conclude that $B.v$, hence also $v$, belongs to the orbit of $e_{d-1}$ under $\G_d$.
\end{itemize}

This completes the proof of the proposition.
\end{proof}

\begin{corollary}\label{cor2}
Theorem \ref{thm3} holds for $d$. In particular, the orbit of $e_{1-d}$ under $\G_d$ is equal to $\mathcal O_d$ from Proposition \ref{prop7}.
\end{corollary}
\begin{proof}
Indeed, by Proposition \ref{prop7}, for any $p \in \A_d$, the orbit of $e_{d-1}$ under $\G_d$ contains $(-1)^{(d-1-p)/2} e_p$.
\end{proof}

We finally prove that $\G_d$ is equal to $\G'_d$ under the assumptions of this subsection.

\begin{proof}
Let $B \in \G'_d$. From Corollary \ref{cor2}, there exists $B_0 \in \G_d$ such that $B_0^{-1}.B$ fixes $e_{1-d}$. This means that $B_0^{-1}.B \in \K_d \subset \G_d$, hence $B$ belongs to $\G_d$.
\end{proof}

%%%%%%%%%%%%%%%%%%%%%%%%%%%%%%%%%%%%%%%%%%%%%%%%%%%%%%%%%%%%%%%%%%%%%%%%%%%%%%%%%%%%
\subsection{The induction step in the even case}
%%%%%%%%%%%%%%%%%%%%%%%%%%%%%%%%%%%%%%%%%%%%%%%%%%%%%%%%%%%%%%%%%%%%%%%%%%%%%%%%%%%

In this section, we assume that $d \geq 4$ is even and that Theorems \ref{thm1}, \ref{thm3} hold for $d-1$.

\begin{proposition}\label{prop8}
The stabilizer of $e_{1-d}$ in $\G_d$ is equal to $\K_d$. It is generated by the $L_p$, $p\in \A_d, p\ne 1-d$.
\end{proposition}
\begin{proof}
As $\G_d \subset \G'_d$, this stabilizer is contained in $\K_d$. Let $\K^{\flat}_d$ be the subgroup of $\K_d$ generated by the $L_p$, $p\in \A_d, p\ne 1-d$. From the induction hypothesis $\G_{d-1} = \G'_{d-1}$ and Proposition \ref{prop3}, we deduce that the image of $\K^{\flat}_d$ under $\varphi_d$ is equal to $\G'_{d-1}$. In order to
conclude that $\K^{\flat}_d$ is equal to $\K_d$, it is sufficient to show, in view of Proposition \ref{prop5}, that the matrix
\begin{equation}\label{eq}
\left( \begin{array}{cc} 1& 2\ ^th^{\star} \\ 0 & \mathbf{1}_{d-1} \end{array} \right)
\end{equation}
belongs to $\K^{\flat}_d$.

For  $p \in \A_d$, $1-d < p < d-1$ define $M_p:= L_{p+2}^{-1} \circ L_p \circ L_{p+2}$.
This element of $\S_{p,p+2}$ satisfies
$$ M_p(e_p) = 2 e_p + e_{p+2}, \quad M_p(e_{p+2}) = -e_p, \quad M_p(e_q) = e_q \quad {\rm if } \ q \ne p,p+2.$$

Let
$$M := M_{d-3} \circ M_{d-5} \circ \ldots \circ M_{3-d}.$$

The matrix $M$ belongs to $\K_d^{\flat}$ and one has
\begin{eqnarray*}
 M(e_{3-d}) &=& e_{d-1} + 2\sum_{p \in \A_d, 1-d < p < d-1} e_p,\\
 M(e_p) &=& -e_{p-2}  \qquad {\rm for} \ p \in \A_d, p> 3-d.
\end{eqnarray*}

Therefore $N:= L_{d-1}^2 \circ M$ satisfies
$$ N(e_{3-d}) = e_{d-1} -2e_{1-d}, \quad N(e_p) = -e_{p-2}  \quad {\rm for} \ p \in \A_d, p> 3-d.$$

It follows that
$$ N^{d-1}(e_p) = e_p + 2 (-1)^{\frac{p-d+1}2} e_{1-d}, \qquad \forall p \in \A_d, p> 1-d.$$

Thus, the inverse of the matrix $N^{d-1}\in\K_d^{\flat}$ has the required form \eqref{eq}.
\end{proof}

\begin{proposition}\label{prop9}
The orbit of $e_{d-1}$ under $\G_d$ is equal to the set of primitive vectors in $\Zset^{\A_d}$ which are congruent mod.$2$ to a vector in $\mathcal E = \{\sum_{p\in \Ad} \bar e_p\}\cup\{\bar e_p\}_{p \in \Ad}$.
\end{proposition}

\begin{proof}
It is clear that vectors in the orbit of $e_{d-1}$ under $\G_d$ are primitive and congruent mod.$2$ to a vector in $\mathcal E$.
\smallskip

Conversely, let $v = \sum_{p\in \A_d} v_p e_p$ be a primitive vector in  $\Zset^{\A_d}$ which is congruent mod.$2$ to a vector in $\mathcal E$.

\begin{lemma}\label{lem3}
If $\sum_{p\in \A_d, p>1-d} (-1)^{\frac{d-1-p}2} v_p =1$, then $v$ belongs to the orbit of $e_{d-1}$ under $\G_d$.
\end{lemma}
\begin{proof}
Let $v' := \sum_{p\in \A_d, p>1-d} v_p e_p$. From the hypothesis of the lemma, $v'$ is a primitive vector, in particular it is not even. Therefore it is congruent mod.$2$ to one of the $e_p$ (with $p\in \A_d,\,p>1-d$) or to $\sum_{p\in \A_d, p>1-d} e_p$. From Theorem
\ref{thm3} for $d-1$, there exists $g \in \G_{d-1}$ such that, after shifting the indices by $1$, $v'$ is the last column of $g$.  By Propositions \ref{prop3} and \ref{prop5}, there exists $B \in \K_d$ such that $v$ is the last column of $B$ (one only needs even multiples of $^t h^{\star}$ because of the congruence condition). As $\K_d$ is contained in $\G_d$ by Proposition \ref{prop8}, we get the assertion of the lemma.
\end{proof}

\begin{lemma}\label{lem4}
If $v_{1-d} = 1$, then $v$ belongs to the orbit of $e_{d-1}$ under $\G_d$.
\end{lemma}
\begin{proof}
For $w= \sum_{p\in \A_d} w_p e_p$, define $\phi(w) := \sum_{p\in \A_d, p>1-d} (-1)^{\frac{d-1-p}2} w_p$. Observe that, for $w \in \Zset^{\A_d}$, $n\in \Zset$, one has
$$ \phi (L_{1-d}^n (w)) = \phi(w) + n w_{1-d}.$$

If $v_{1-d} =1$, there exists $n \in \Zset$ such that $L_{1-d}^n(v)$ satisfies the hypothesis of Lemma \ref{lem3}. Then $L_{1-d}^n(v)$ belongs to the orbit of $e_{d-1}$ under $\G_d$, and the same is true for $v$.
\end{proof}

\begin{lemma}\label{lem5}
There exists $g \in \G_d$ such that the first component of $g.v$ is equal to one.
\end{lemma}
\begin{proof}
The argument is by infinite descent. It is clear that there exists $g \in \G_d$ such that the first component of $g.v$ is positive. Then we may assume (replacing $v$ by an appropriate $g.v$) that $v_{1-d} >0$ and that, for any $g \in \G_d$, the first component of $g.v$ is either $\leq 0$ or $\geq v_{1-d}$. We have to show that $v_{1-d} =1$. We assume by contradiction that $v_{1-d} >1$.
\smallskip

As $v$ is primitive, there exists $p \in \A_d$, $p>1-d$,  such that $v_p$ is not a multiple of $v_{1-d}$. Let $\bar v_{1-d}>0$ be the smallest common divisor of $v_p, v_{1-d}$. One has $1 \leq \bar v_{1-d} < v_{1-d}$. By Lemma \ref{lem2}, there exists
an element $g$ in the subgroup $\S_{1,p}$ generated by $L_1,L_p$ such that the first coordinate of $g.v$ is equal to $\bar v_{1-d}$. This gives the required contradiction.
\end{proof}

The desired proposition follows from Lemmas \ref{lem4} and \ref{lem5}.
\end{proof}

Similarly to the previous subsection, the induction step for $d$ even follows from Propositions \ref{prop8} and \ref{prop9}.

At this point, the inductive proofs of Theorems \ref{thm1}, \ref{thm2}, \ref{thm3} are now complete.

\section{Dehn twists and hyperelliptic Rauzy diagrams}\label{s.A'Campo}

In this section, we give an alternative proof of the precise version of Theorem \ref{t.AMY-intro} stated as  Theorem \ref{thm1} above. For this sake, we start with a general discussion of Dehn twists arising naturally from certain loops in Rauzy diagrams and then we specialize this discussion to the case of hyperelliptic Rauzy diagrams.

\subsection{General remarks on Rauzy diagrams and Dehn twists}\label{secDehn}

Once again, we assume some familiarity with the reference \cite{Y-Pisa} during this entire subsection.

Let $\A$ be an alphabet with $d \geq 2$ letters, let $\R$ be an arbitrary Rauzy class on $\A$, and let $\D$ be the associated Rauzy diagram.

\subsubsection{The surfaces $M_\pi$ and their decorations}

A partial reference for what follows is \cite[Section 9.2]{Y-Pisa}.

For every $\pi \in \R$, we construct a \emph{canonical} translation surface $M_{\pi}$
with combinatorial data $\pi$ whose length data $\lambda^{can}$ and suspension data $\tau^{can}$ are given by
$$ \lambda^{can}_\alpha = 1, \tau^{can}_\alpha = \pi_b(\alpha) - \pi_t(\alpha), \quad \forall \alpha \in \A.$$
Denote by $g$ the genus of $M_\pi$ and by $s$ the cardinality of $\Sigma_\pi$. Recall that both $g$ and $s$ depend only on $\R$ and $d = 2g +s-1$.

The surface is obtained by identifying parallel sides of a polygon $P_\pi$ whose leftmost vertex, denoted by $U_0$ or $V_0$, is at $0 \in \Cset$. The rightmost vertex, denoted by $U_d$ or $V_d$, is at $d$. The vertices above the real axis are denoted (from left to right) by $U_1, \ldots, U_{d-1}$. The vertices below the real axis are denoted (from left to right) by $V_1, \ldots, V_{d-1}$. As $\sum \tau_\alpha = 0$, Veech's zippered rectangle construction is not needed here.

We denote by $\Sigma_\pi$ the set of marked points of $M_\pi$, we  equip $M_\pi$ with a basepoint $\ast_{\pi}=1/2\in\mathbb{C}$ and we set $O_\pi = d/2 \in \Cset$.   We denote by $\mathcal T\not\ni \ast_{\pi}$ a curvilinear triangle whose sides are a curvilinear ``vertical'' segment $\eta=[U_{d-1},V_{d-1}]$ through $O_\pi$ and the sides $[U_{d-1},U_d]$, $[V_{d-1},V_d]$ of $P_\pi$.

We denote by $\Sigma^*_\pi$ the subset of $M_\pi$ consisting of $O_\pi$ and the midpoints of the sides of $P_\pi$. Its cardinality is equal to $d+1$.

For each $\alpha \in \A$, we define an oriented loop $\theta_\alpha$ in $M_\pi \setminus \Sigma_\pi$, based at $\ast_\pi$:
\begin{itemize}
\item We choose a simple path $\theta^t_\alpha$ (resp.  $\theta^b_\alpha$ ) from $\ast_\pi$ to the middle point of the top (resp. bottom) $\alpha$-side of $P_\pi$ passing through $O_{\pi}$ via the horizontal segment $[\ast_{\pi},O_{\pi}]$; this path is contained in the interior of $P_\pi$ except for its endpoint.
\item We ask that the  $\theta^\va_\alpha$, $\alpha \in \A$, $\va \in \{t,b \}$ are disjoint
except from their endpoints and $[\ast_{\pi}, O_{\pi}]$, and also disjoint from $[U_{d-1},V_{d-1}]$ except at $O_\pi$.
\item $\theta_\alpha$ is the concatenation of $\theta^t_\alpha$ and $(\theta^b_\alpha)^{-1}$ (so that $\theta_\alpha$ is oriented upwards).
\end{itemize}

The difference $M_\pi \setminus \cup_{\alpha \in \A} \theta_\alpha$ is a finite union of open disks. Each of this disks contains exactly one point of $\Sigma_\pi$.

Recall that the fundamental group $\pi_1(M_\pi \setminus \Sigma_\pi,\ast_\pi)$ is a free group on $d = 2g + s -1$ generators, namely, the classes of the $\theta_\alpha$, $\alpha \in \A$: see \cite[Subsection 4.5]{Y-Pisa}, for instance.

\subsubsection{The homeomorphisms $H_\gamma$}

Consider an arrow $\gamma: \pi \to \pi'$ of $\D$. We claim that one can naturally associated to the arrow $\gamma$ an orientation-preserving homeomorphism $H_\gamma : (M_\pi, \Sigma_\pi \cup \Sigma^*_\pi\cup\{\ast_{\pi}\}) \to (M_{\pi'}, \Sigma_{\pi'} \cup \Sigma^*_{\pi'}\cup \{\ast_{\pi'}\})$ which is uniquely defined modulo isotopy (amongst homeomorphisms sending $\Sigma_\pi \cup \Sigma^*_{\pi}\cup\{\ast_{\pi}\}$ to $\Sigma_{\pi'} \cup \Sigma^*_{\pi'}\cup\{\ast_{\pi'}\}$).

The homeomorphism $H_\gamma$ is constructed as follows. We denote by $\alpha_t, \alpha_b$ the letters of $\A$ such that $\pi_t (\alpha_t) = \pi_b (\alpha_b) = d$ and we let $B_\gamma \in SL(\Zset^{\A})$ be the matrix associated to $\gamma$ by the Rauzy--Veech algorithm / KZ-cocycle. We assume that $\gamma$ is of top type (as the bottom case is completely similar).

\begin{itemize}
%\item The first step is to construct a deformation $P^1_\pi$ of $P_\pi$, leading to a deformation $M^1_\pi$ of $M_\pi$, from the length data $\lambda^1:= ^t \negthickspace B_\gamma \lambda^{can}$, $\tau^1:= \tau^{can}$. (As $\sum_\alpha \tau^1_\alpha =0$, the polygon is not self-intersecting, hence, as for $P_\pi$, the zippered rectangles construction of Veech is unnecessary ). We also choose an orientation-preserving homeomorphism $h_1$ from $P_\pi$ onto $P^1_\pi$ which sends vertices to vertices, top $\alpha$-sides to top $\alpha$-sides, bottom $\alpha$-sides on bottom $\alpha$-sides, and is compatible with the identification of bottom and top $\alpha$-sides (hence $h_1$ defines a homeomorphism from $M_\pi$ onto $M^1_\pi$).

\item We cut the triangle $\mathcal T$ from $P_\pi$ along $\eta$ and glue\footnote{If $\gamma$ were of bottom type,  we would glue $\mathcal T$ to the top $\alpha_b$-side of $P_\pi$.} it again, after the appropriate translation, through the identification of  the bottom $\alpha_t$-side of $P_\pi$ and the side $[U_{d-1},U_d]$ of $\mathcal T$. We obtain in this way a  polygon $P^0_{\pi'}$ with a pair of curvilinear ``vertical'' sides. This cutting and glueing process corresponds to the basic step of the Rauzy--Veech algorithm. The sides of $P^0_{\pi'}$ are labelled by $\A$ from $0$ in the same cyclical order than for $P_{\pi'}$. In particular, the curvilinear ``vertical'' sides are labelled by $\alpha_t$.
The surface $M^0_{\pi'}$ obtained from $P^0_{\pi'}$ by glueing bottom and top sides of the same name is \emph{canonically isomorphic} to $M_\pi$.

\item We choose an orientation-preserving homeomorphism $h$ from $P^0_{\pi'}$ onto $P_{\pi'}$ with the following properties
\begin{itemize}
\item For each $\alpha \in \A$, $h$ sends the top $\alpha$ side of $P^0_{\pi'}$  onto the top $\alpha$ side of $P_{\pi'}$, and the bottom $\alpha$ side of $P^0_{\pi'}$  onto the bottom $\alpha$ side of $P_{\pi'}$. This is done in a way which is compatible with the identification of top and bottom sides in $P^0_{\pi'}$ and $P_{\pi'}$.
\item For each $\alpha \in \A$, $\alpha \ne \alpha_t$, $h$ sends the midpoint of the top (resp.bottom) $\alpha$-side of $P^0_{\pi'}$ to the midpoint of the top (resp. bottom) $\alpha$-side of $P_{\pi'}$.
\item $h$ sends the point $O_\pi$ (on the top $\alpha_t$-side of $P^0_{\pi'}$) to the
 midpoint of the top $\alpha_t$-side of $P_{\pi'}$.
 \item $h$ sends the midpoint of the bottom  $\alpha_t$-side of $P_\pi$ (which lies inside $P^0_{\pi'}$) to $O_{\pi'}$ and $h$ sends $\ast_{\pi}$ to $\ast_{\pi'}$.
 \end{itemize}
\item Finally, $H_\gamma: M_\pi \equiv M^0_{\pi'} \to M_{\pi'} $ is the homeomorphism induced by $h$.
\end{itemize}
The reader can check that the homotopy class of $H_\gamma$ (mod $\Sigma_\pi \cup \Sigma^*_\pi\cup\{\ast_{\pi}\}$) does not depend on the choices of $h$.

\begin{figure}[h!]
%\centering
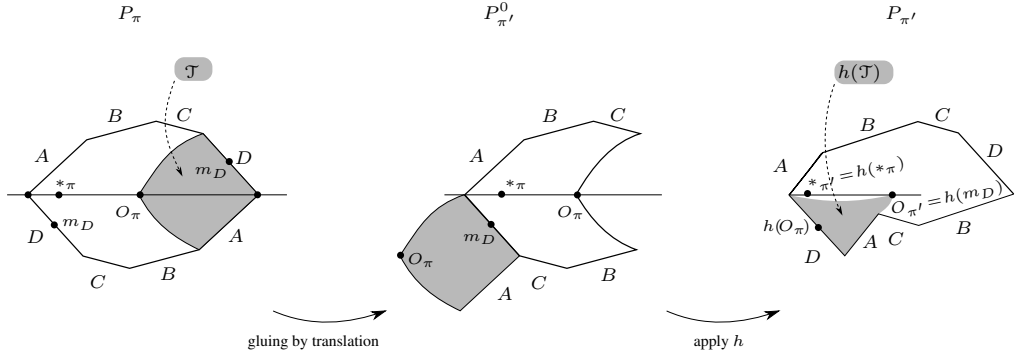\caption{Construction of the homeomorphism $H_{\gamma}$.}
\end{figure}

\subsubsection{Naming the marked points}

Let $\pi \in \R$. For each marked point $a$ in $\Sigma_\pi$, define $\A(\pi,a)\subset \A$ as the subset of letters $\alpha \in \A$ such that $a$ is the left endpoint of the
$\alpha$-sides of $P_\pi$. We have a partition
\begin{equation}\label{e.partition}
\A = \bigsqcup_{a \in \Sigma_\pi} \A(\pi, a).
\end{equation}

It is easy to check that, for any arrow $\gamma:\pi \to \pi'$ of $\D$, the homeomorphism $H_\gamma$ constructed above satisfies\footnote{This is somewhat related to \cite[Proposition 7.7]{Y-Pisa}.}, for any $a \in \Sigma_\pi$
$$ \A( \pi, a) = \A (\pi', H_{\gamma}(a)).$$

In other terms, the partition \eqref{e.partition} above depends only on $\R$, not on $\pi$. Therefore, we can use \eqref{e.partition} to name in a consistent way the points of the various $\Sigma_\pi$, $\pi \in \R$. The homeomorphisms $H_\gamma$ respect the naming.

\begin{remark}
On the other hand, when $\gamma$ is a loop, (i.e $\pi = \pi'$), the homeomorphism $H_\gamma$ permutes in a non trivial way the points of $\Sigma^*_\pi$.
\end{remark}

\subsubsection{The groupoid $\Gamma(\D)$}

Consider the non-oriented Rauzy diagram $\wt\D$ associated to $\D$: it has the same vertices than $\D$; for each arrow $\gamma: \pi \to \pi'$ of $\D$, there are two arrows
$\gamma^+: \pi \to \pi'$ and $\gamma^-: \pi' \to \pi$ in $\wt\D$.

We define $\Gamma(\D)$ as the groupoid of reduced oriented paths in $\wt\D$ (i.e., the groupoid of oriented paths quotiented by the cancellation rules $\gamma^+ \star \gamma^- =\gamma^- \star \gamma^+ =1$).

To each arrow $\gamma^+:\pi \to \pi'$ of positive type of $\wt\D$, we have constructed above a isotopy class $[H_\gamma]$ from $(M_\pi, \Sigma_\pi \cup \Sigma^*_\pi)$ to $(M_{\pi'}, \Sigma_{\pi'} \cup \Sigma^*_{\pi'})$ rel. $\Sigma_\pi \cup \Sigma^*_\pi$ which respects the naming of the points of $\Sigma_\pi$, $\Sigma_{\pi'}$. To an arrow $\gamma^-$ of negative type,
we associate the isotopy class of $H_\gamma^{-1}$. Compare with \cite[Section 9.2]{Y-Pisa}.

We also define a groupoid ${\rm Mod}(\R)$ in the following way. Its vertices are the elements of $\R$. The set ${\rm Mod}(\pi,\pi')$ of arrows from a vertex $\pi$ to a vertex $\pi'$ consists of the isotopy classes of orientation-preserving homeomorphisms  from $(M_\pi, \Sigma_\pi \cup \Sigma^*_\pi\cup\{\ast_{\pi}\})$ to $(M_{\pi'}, \Sigma_{\pi'} \cup \Sigma^*_{\pi'}\cup\{\ast_{\pi'}\})$ rel. $\Sigma_\pi \cup \Sigma^*_\pi\cup\{\ast_{\pi}\}$  which respects the naming of the points of $\Sigma_\pi$, $\Sigma_{\pi'}$. In particular, the image of $\MP:= {\rm Mod}(\pi,\pi)$ under the ``forget $\Sigma^*_\pi\cup\{\ast_{\pi}\}$'' homomorphism is the \emph{pure} mapping class group of $(M_\pi, \Sigma_{\pi})$.

We extend the map $\gamma^+ \mapsto [H_\gamma],\; \gamma^- \mapsto [H_\gamma^{-1}]$ to a morphism of groupoids from $\Gamma(\D)$ to ${\rm Mod}(\R)$. In particular, for each $\pi \in \R$, we have a group homomorphism from the fundamental group $\pi_1(\wt\D,\pi)$ to the pure modular group $\MP$ of $(M_\pi, \Sigma_{\pi})$.

\begin{question} What is the image of this homomorphism from $\pi_1(\wt\D,\pi)$ to $\MP$?
\end{question}

\begin{remark} We give an answer to this question for hyperelliptic Rauzy diagrams in Subsection \ref{ss.A'Campo} below.
\end{remark}

\subsubsection{Action of $H_\gamma$ on the fundamental groups}
Let $\gamma: \pi \to \pi'$ be an arrow of $\D$. We compute the homomorphism $\pi_1(\gamma): \pi_1(M_\pi \setminus \Sigma_\pi, \ast_\pi) \to \pi_1(M_{\pi'} \setminus \Sigma_{\pi'}, \ast_{\pi'})$ induced by $H_\gamma$. We denote by $\alpha_w$ the winner of $\gamma$, by $\alpha_{\ell}$ the loser of $\gamma$. Recall the generators $\theta_\alpha$, $\alpha \in \A$  of $\pi_1(M_\pi \setminus \Sigma_\pi, \ast_\pi)$. The corresponding generators for $\pi_1(M_{\pi'} \setminus \Sigma_{\pi'}, \ast_{\pi'})$ are denoted by $\theta'_\alpha$, $\alpha \in \A$. A direct inspection of our construction shows that:

\begin{proposition}\label{p.Hgamma}
One has $\pi_1(\gamma)(\theta_\alpha) = \theta'_\alpha$, for $\alpha \ne \alpha_{\ell}$, and
$$  \pi_1(\gamma)(\theta_{\alpha_{\ell}}) = \left \{ \begin{array}{ll} \theta'_{\alpha_{\ell}} \star (\theta'_{\alpha_{w}})^{-1} & \text {if $\gamma$ is of top type} \\ &\\(\theta'_{\alpha_{w}})^{-1} \star \theta'_{\alpha_{\ell}} & \text { if $\gamma$ is of bottom type} \end{array} \right.$$
\end{proposition}

\subsubsection{Pure cycles in $\D$}

A simple \emph{oriented} loop in $\D$ is called a {\it pure cycle} if all its arrows have the same type (bottom or top). Equivalently, all its arrows have the same winner.

Let $\pi \in \R$. There are exactly two pure cycles through $\pi$. One is made of arrows of top type, with winner $\alpha_t$. Its length is $d- \pi_b(\alpha_t)$. The other is made of arrows of bottom type, with winner   $\alpha_b$. Its length is $d- \pi_t(\alpha_b)$.

In the next proposition, Dehn twists in $M_\pi$ along the curves $\theta_\alpha$, $\alpha \in \A$ are considered as elements of $\MP$ by choosing a representative which is supported in a neighborhood of $\theta_\alpha$ and \emph{exchanges} $O_\pi$ and the midpoint of the $\alpha$-sides of $P_\pi$ (these two points are the only points of $\Sigma^*_\pi$ lying on $\theta_\alpha$).
\begin{proposition}\label{p.purecycles}
Let $\pi \in \R$ and let $\Gamma$ be a pure cycle through $\pi$.
If $\Gamma$ is of top type, the image of $\Gamma \in \pi_1(\wt\D,\pi)$ in $\MP$ is the left Dehn twist along  $\theta_{\alpha_t}$.
If $\Gamma$ is of bottom type, the image of $\Gamma$ in  $\MP$ is the right Dehn twist along  $\theta_{\alpha_b}$.
\end{proposition}
\begin{proof}
This fact can be deduced by direct inspection on $M_{\pi}$ or by  computing the action on fundamental groups (in a similar way to Proposition \ref{p.Hgamma} above).
\end{proof}

\subsection{Hyperelliptic Rauzy diagrams, Dehn twists and braid groups}\label{ss.A'Campo} In this subsection, we restrict ourselves to hyperelliptic Rauzy diagrams. Let $d\geq 2$ be an integer and consider the hyperelliptic Rauzy class $\Rd$ equipped of its central vertex $\pi^* = \pi^*(d)$, and the hyperelliptic Rauzy diagram $\Dd$ introduced in Subsection \ref{ss.hypRauzy}.

\subsubsection{Elementary simple loops in $\Dd$} The canonical surface $M_{\pi^*}$ is hyperelliptic thanks to the hyperelliptic involution $\tau_{\pi^*}$ given by central symmetry at the point $O_{\pi^*}=d/2$. In particular, $O_{\pi^*}=d/2$ and the midpoints of the sides of the polygon $P_{\pi^*}$ contains Weierstrass points of $M_{\pi^*}$. Moreover, the marked point of $M_{\pi^*}$ is a Weierstrass point when $d$ is even, while the marked points of $M_{\pi^*}$ are exchanged by the hyperelliptic involution.

In this setting, Proposition \ref{p.purecycles} says that the image in $\textrm{Mod}(\pi^*)$ of a non-oriented loop $\gamma'\in\pi_1(\wt\Dd,\pi^*)$ associated to an elementary simple loop in $\Rd$ is a Dehn twist exchanging $O_{\pi^*}$ with the midpoint of a side of $P_{\pi^*}$. In other words, we have computed the action of $\gamma'$ at the homotopical level (as promised in Remark \ref{r.homology-homotopy}): again, it  depends only on the type and winner of $\gamma$, and the homotopical action of a loop of bottom type is the inverse of the homotopical action of a loop of top type and the same winner.

\begin{remark} The fact that elementary simple loops in $\Rd$ act by Dehn twists implies that one can not expect the hyperelliptic Rauzy--Veech groups $\mathfrak G_d$ to coincide with the full symplectic group in Theorem \ref{t.AMY-intro}. More precisely, Dehn twists act on homology by symplectic transvections, so that $\mathfrak G_d$ is generated by $d=2g$ symplectic transvections when $d$ is even. However, it is known that one can not generate $\textrm{Sp}(2g,\mathbb{Z})$ with fewer than $2g+1$ symplectic transvections (cf. \cite[Proposition 6.5]{FM}).
\end{remark}

\subsubsection{Symmetric mapping class groups and braid groups} The Dehn twists associated to the elementary simple loops in $\Rd$ commute with the hyperelliptic involution $\tau_{\pi^*}$ of $M_{\pi^*}$. Therefore, the image of the homomorphism from $\pi_1(\wt\Dd,\pi^*)$ to ${\rm Mod}(\pi^*)$ is contained in the \emph{symmetric mapping class subgroup}, that is, the centralizer of $\tau_{\pi^*}$ in ${\rm Mod}(\pi^*)$.

It follows that the Dehn twists associated to elementary simple loops in $\Rd$ are lifts to $M_{\pi^*}$ of certain elements of a \emph{braid group}\footnote{The braid group $B_m$ is the fundamental group of the space $\mathbb{C}^{\langle m\rangle}$ of configurations of finite subsets of $\mathbb{C}$ of cardinality $m$ based at an arbitrarily fixed configuration $\ast\in \mathbb{C}^{\langle m\rangle}$.}. More precisely, the hyperelliptic translation surface $M_{\pi^*}$ can be thought of as the hyperelliptic Riemann surface $y^2=(x-b_1)\dots(x-b_{d+1})$ equipped with the Abelian differential $dx/y$ (whose zeroes are at the points at infinity) for an appropriate choice of configuration $\{b_1,\dots, b_{d+1}\}$ of pairwise distinct points in $\mathbb{C}$. Here, the subset $\{b_1,\dots, b_{d+1}\}$ of Weierstrass points correspond to the set of $\Sigma^*_{\pi^*}$, and, for the sake of concreteness, we make our choices so that $b_{d+1}$ corresponds to $O_{\pi^*}$ while $b_n$, $1\leq n\leq d$ correspond to the midpoints of sides of $P_{\pi^*}$. In this context, we see that the Dehn twists associated to elementary simple loops in $\Rd$ are lifts to $M_{\pi^*}$ of the elements $\theta_n$, $1\leq n\leq d$, of the braid group $B_{d+1}$ exchanging $b_n$ and $b_{d+1}$.

Note that $\{\theta_n: 1\leq n\leq d\}$ is not a system of Artin\footnote{See \cite[Section 9.2]{FM}, for instance.} standard generators $\{\sigma_j:1\leq j\leq d\}$ where $\sigma_j$ exchanges $b_j$ and $b_{j+1}$, but it is not hard to see that $\sigma_j$ can be written in terms of $\theta_j$ and $\theta_{j+1}$ (by conjugation). In particular, $\{\theta_n: 1\leq n\leq d\}$ generates the braid group $B_{d+1}$ and, \emph{a fortiori}, the image of the homomorphism from $\pi_1(\wt\Dd,\pi^*)$ to ${\rm Mod}(\pi^*)$ is precisely the symmetric mapping class group ${\rm SMod}(\pi^*)\simeq B_{d+1}$.

\begin{remark} The symmetric mapping class group ${\rm SMod}(\pi^*)$ is an infinite-index subgroup of ${\rm Mod}(\pi^*)$ corresponding to the orbifold fundamental group\footnote{An interesting consequence of this fact is the non-connectedness of the hyperelliptic Teichm\"uller spaces.} of projectivized hyperelliptic connected components of the moduli spaces of translation surfaces: see, e.g., Looijenga--Mondello \cite{LM}. Thus, we have just shown that, in a certain sense, the hyperelliptic Rauzy diagrams ``see'' the topology of the projectivized hyperelliptic connected components of the moduli spaces of translation surfaces.
\end{remark}

\subsubsection{Monodromy representations of braid groups}

The elements of ${\rm SMod}(\pi^*)$ act on the homology $M_{\pi^*}$. This induces a natural monodromy representation
$$\rho_{d+1}:B_{d+1}\to \textrm{Sp}(H_1(M_{\pi^*},\mathbb{Z}))$$
of the braid group $B_{d+1}$.

It follows from our discussion above of the homomorphism from $\pi_1(\wt\Dd,\pi^*)$ to ${\rm Mod}(\pi^*)$ that the hyperelliptic Rauzy-Veech group $\Gd$ coincides with the image $\rho_{d+1}(B_{d+1})$ of the monodromy representation $\rho_{d+1}$.

As it turns out, the image of $\rho_{d+1}$ was described by A'Campo \cite[Th\'eor\`eme 1]{A'Campo}:
\begin{theorem}[A'Campo]\label{t.A'Campo} Let $d\geq 2$. The image of $\rho_{d+1}$ contains the congruence subgroup of level two of $\textrm{Sp}(H_1(M_{\pi^*},\mathbb{Z}))$. Moreover, the reduction of $\rho_{d+1}(B_{d+1})$ mod.$2$ is isomorphic to a symmetric group of order $d+1$, resp. three, for $d\neq 3$, resp. $d=3$.
\end{theorem}

\begin{remark} Notice that Theorem \ref{t.A'Campo} only describes the action on on the absolute homology, while Corollary \ref{cor1} describes the action on the full relative homology. For this reason, we get slightly different groups in the special case $d=3$.
\end{remark}

In this way, we recover the description of the hyperelliptic Rauzy-Veech group $\Gd = \rho_{d+1}(B_{d+1})$ in Theorems \ref{t.AMY-intro} and \ref{thm1}.

\appendix

\section{A pinching and twisting group with small Zariski closure}\label{a.pinching+twisting-Zariski}

Let $\rho$ be the third symmetric power of the standard representation of $SL(2,\mathbb{R})$. In concrete terms, $\rho$ is constructed as follows. Consider the basis $\mathcal{B} = \{X^3, X^2Y, XY^2, Y^3\}$ of the space $V$ of homogenous polynomials of degree $3$ on two variables $X$ and $Y$. By letting $g=\left(\begin{array}{cc} a & b \\  c & d \end{array}\right)\in SL(2,\mathbb{R})$ act on $X$ and $Y$ as $g(X)=aX+cY$ and $g(Y)=bX+dY$, we obtain an induced action $\rho(g)$ on $V$ whose matrix in the basis $\mathcal{B}$ is
$$\rho\left(\begin{array}{cc} a & b \\ c & d \end{array}\right) = \left(\begin{array}{cccc} a^3 & a^2 b & a b^2 & b^3 \\
3 a^2 c & a^2 d + 2 a b c & b^2 c + 2 a b d & 3 b^2 d \\ 3 a c^2 & b c^2 + 2 a c d & a d^2 + 2 b c d & 3 b d^2 \\
c^3 & c^2 d & c d^2 & d^3 \end{array}\right)$$

Note that the faithful representation $\rho$ is the unique irreducible four-dimensional representation of $SL(2,\mathbb{R})$. Furthermore, the matrices $\rho(g)$ preserve the symplectic structure on $V$ associated to the matrix
$$J = \left(\begin{array}{cccc} 0 & 0 & 0 & -1 \\ 0 & 0 & 1/3 & 0 \\
0 & -1/3 & 0 & 0 \\ 1 & 0 & 0 & 0 \end{array}\right)$$
Indeed, a direct calculation shows that if $g=\left(\begin{array}{cc} a & b \\ c & d \end{array}\right)$, then
$${}^t\rho(g)\cdot J\cdot \rho(g) = \left(\begin{array}{cccc} 0 & 0 & 0 & -(a d - b c)^3 \\
0 & 0 & \frac{(a d - b c)^3}{3} & 0 \\  0 & - \frac{(a d - b c)^3}{3} & 0 & 0 \\ (a d - b c)^3 & 0 & 0 & 0 \end{array}\right)$$
where ${}^t\rho(g)$ stands for the transpose of $\rho(g)$.

Denote by $\mathcal{M}$ the group generated by the matrices
$$A = \rho\left(\begin{array}{cc} 1 & 1 \\ 0 & 1 \end{array}\right) = \left(\begin{array}{cccc} 1 & 1 & 1 & 1 \\ 0 & 1 & 2 & 3 \\ 0 & 0 & 1 & 3 \\
0 & 0 & 0 & 1\end{array}\right) \quad \textrm{and} \quad B = \rho\left(\begin{array}{cc} 1 & 0 \\ 1 & 1 \end{array}\right) = \left(\begin{array}{cccc} 1 & 0 & 0 & 0 \\ 3 & 1 & 0 & 0 \\ 3 & 2 & 1 & 0 \\
1 & 1 & 1 & 1\end{array}\right)$$

On one hand, the group $\mathcal{M}$ has small Zariski closure.

\begin{proposition}\label{p.-Zariski} The group $\mathcal{M}$ is not Zariski dense in $Sp(V)$.
\end{proposition}

\begin{proof} Note that $\rho$ is a \emph{polynomial}\footnote{That is, the entries of $\rho(g)\in Sp(V)$ depend polynomially on the entires of
$g\in SL(2,\mathbb{R})$.} homomorphism from $SL(2,\mathbb{R})$ to $Sp(V)$. In particular, it follows that the image $H=\rho(SL(2,\mathbb{R}))$ of $\rho$ is Zariski closed in $Sp(V)$: see, e.g., Corollary 4.6.5 in Witte-Morrris book \cite{WM}. Since  $SL(2,\mathbb{Z})$ is a Zariski dense subgroup of $SL(2,\mathbb{R})$ generated by $\left(\begin{array}{cc} 1 & 1 \\ 0 & 1 \end{array}\right)$ and $\left(\begin{array}{cc} 1 & 0 \\ 1 & 1 \end{array}\right)$, we see that the Zariski closure of the group $\mathcal{M}$ is $H=\rho(SL(2,\mathbb{R}))$ ($\simeq SL(2,\mathbb{R})$ because $\rho$ is faithful). This proves the proposition (since $H$ is a proper linear algebraic subgroup of $Sp(V)$).
\end{proof}

On the other hand, the group $\mathcal{M}$ is pinching and twisting in the sense of Avila--Viana (see \cite{AV} and \cite[Section 2]{MMY}):

\begin{proposition}\label{p.pinching+twisting} The matrix $A.B\in\mathcal{M}$ is a pinching element\footnote{Its eigenvalues are all real with distinct moduli.} and the matrix $A\in\mathcal{M}$ is twisting\footnote{$A(F)\cap F'=\{0\}$ for all $A.B$-invariant isotropic subspaces $F\subset V$ and all $A.B$-invariant coisotropic subspaces $F'\subset V$ with $\textrm{dim}(F)+\textrm{dim}(F')=4$.} with respect to the pinching element $A.B\in\mathcal{M}$.
\end{proposition}

\begin{proof} The first assertion follows from the fact that
$$9+4\sqrt{5} > \frac{3+\sqrt{5}}{2} > \frac{3-\sqrt{5}}{2} > \frac{1}{9+4\sqrt{5}}$$
are the eigenvalues of $A.B\in\mathcal{M}$.

The second assertion is established by the following reasoning. The columns of
$$M = \left(\begin{array}{cccc} -\frac{1}{4} + \frac{(9 + 4 \sqrt{5})}{4} & 1 - \frac{(3 + \sqrt{5})}{2} & 1 - \frac{(3 - \sqrt{5})}{2} &
  -\frac{1}{4} + \frac{(9 - 4 \sqrt{5})}{4} \\
  \frac{9}{8} + \frac{3(9 + 4 \sqrt{5})}{8} & -2 + \frac{(3 + \sqrt{5})}{2} &
   -2 + \frac{(3 - \sqrt{5})}{2} &  \frac{9}{8} + \frac{3(9 - 4 \sqrt{5})}{8} \\
   -\frac{15}{8} + \frac{3(9 + 4 \sqrt{5})}{8} &  \frac{(3 + \sqrt{5})}{2} &
  \frac{(3 - \sqrt{5})}{2} & -\frac{15}{8} + \frac{3(9 - 4 \sqrt{5})}{8} \\
   1 &   1 & 1 & 1 \end{array}\right)$$
consist of eigenvectors of $A.B$. Thus, $T=M^{-1}\cdot A\cdot M$ is the matrix of $A$ in the corresponding basis of eigenvectors of $A.B$. By definition, $A$ is twisting with respect to $A.B$ when all entries of $T$ and all of its $2\times 2$ minors associated to Lagrangian planes are non-zero. As it turns out, this last property holds because a direct computation reveals that $T$ and the matrix $T^{\wedge 2}$ of $2\times 2$ minors are given by:

$$
T = \left(\begin{array}{cccc} \frac{8(5 + 2 \sqrt{5})}{25} & \frac{2(5 + 3 \sqrt{5})}{25} &
  \frac{(5 + \sqrt{5})}{25} & \frac{1}{(
  5 \sqrt{5})} \\
  -\frac{6(5 + 3 \sqrt{5})}{25} & \frac{2(5 + \sqrt{5})}{25} & \frac{7}{5 \sqrt{5}} & -\frac{3(-5 + \sqrt{5})}{25} \\
  \frac{3(5 + \sqrt{5})}{25} & -\frac{7}{5 \sqrt{5}} & -\frac{2(-5 + \sqrt{5})}{25} & \frac{6(-5 + 3 \sqrt{5})}{25}
  \\ -\frac{1}{5 \sqrt{5}} & \frac{(5 - \sqrt{5})}{25} &
  \frac{2}{5} - \frac{6}{5 \sqrt{5}} & -\frac{(8(-5 + 2 \sqrt{5})}{25} \end{array}\right)
$$
and

$$ T^{\wedge 2} = \left(\begin{array}{cccccc}
\frac{56}{25} + \frac{24}{5 \sqrt{5}} & \frac{32}{25} + \frac{16}{5 \sqrt{5}} &
  \frac{18}{25} + \frac{6}{5 \sqrt{5}} & \frac{6}{25} + \frac{2}{5 \sqrt{5}} & \frac{2}{25} + \frac{2}{5 \sqrt{5}} &
  \frac{1}{25} \\

-\frac{32}{25} - \frac{16}{5 \sqrt{5}} & \frac{6}{25} + \frac{2}{5 \sqrt{5}} &
  \frac{9}{25} + \frac{9}{5 \sqrt{5}} & \frac{3}{25} + \frac{3}{5 \sqrt{5}} & \frac{11}{25} & -\frac{2}{25} + \frac{2}{5 \sqrt{5}}  \\

  \frac{6}{25} + \frac{2}{5 \sqrt{5}} & -\frac{3}{25} - \frac{3}{5 \sqrt{5}} & \frac{13}{25} & -\frac{4}{25} & -\frac{3}{25} + \frac{3}{5 \sqrt{5}} &
  \frac{6}{25} - \frac{2}{5 \sqrt{5}} \\
  \frac{18}{25} + \frac{6}{5 \sqrt{5}} & -\frac{9}{25} - \frac{9}{5 \sqrt{5}} & -\frac{36}{25} & \frac{13}{25} & -\frac{9}{25} + \frac{9}{5 \sqrt{5}} &
  \frac{18}{25} - \frac{6}{5 \sqrt{5}} \\

  -\frac{2}{25} - \frac{2}{5 \sqrt{5}} & \frac{11}{25} &
  \frac{9}{25} - \frac{9}{5 \sqrt{5}} & \frac{3}{25} - \frac{3}{5 \sqrt{5}} &
  \frac{6}{25} - \frac{2}{5 \sqrt{5}} & -\frac{32}{25} + \frac{16}{5 \sqrt{5}} \\

  \frac{1}{25} &
  \frac{2}{25} - \frac{2}{5 \sqrt{5}} & \frac{18}{25} - \frac{6}{5 \sqrt{5}} & \frac{6}{25} - \frac{2}{5 \sqrt{5}} &
  \frac{32}{25} - \frac{16}{5 \sqrt{5}} & \frac{56}{25} - \frac{24}{5 \sqrt{5}} \end{array}\right)
$$
\end{proof}

In summary, these propositions say that $\mathcal{M}$ is the desired group: it is pinching and twisting, but not Zariski dense in $Sp(V)$.

\begin{remark} Observe that the group $\rho(SL(2,\mathbb{Z}))$ does not contain Galois-pinching\footnote{Pinching elements whose characteristic polynomials have the largest possible Galois group among reciprocal integral  polynomials (namely, hyperoctahedral groups).} elements of $Sp(V)$ in the sense of \cite{MMY} because $H=\rho(SL(2,\mathbb{R}))$ has rank $1$. Alternatively, this fact can be shown as follows. A straightforward computation reveals that the characteristic polynomial of $\rho(g)$ is
$$(x^2-\textrm{tr}(g)\det(g)x+\det(g)^3)\cdot (x^2 - \textrm{tr}(g)(\textrm{tr}(g)^2 - 3\det(g))x + \det(g)^3)$$
and, consequently, the eigenvalues of $\rho(g)$ are
$$\frac{1}{2}\det(g)\left(\textrm{tr}(g)\pm\sqrt{\textrm{tr}(g)^2 - 4 \det(g)}\right),$$
and
$$\frac{1}{2}\left(\textrm{tr}(g)(\textrm{tr}(g)^2 - 3\det(g)) \pm (\textrm{tr}(g)^2 - \det(g)) \sqrt{\textrm{tr}(g)^2 - 4 \det(g)}\right).$$
Therefore, the Galois group of the characteristic polynomial $\rho(g)$, $g\in SL(2,\mathbb{Z})$, is not the largest possible among reciprocal polynomials of degree four.
\end{remark}

\begin{remark} It seems \emph{unlikely} to find pinching and twisting monoids of symplectic matrices which are not Zariski dense in the context of the Kontsevich--Zorich cocycle. Indeed, Filip's classification theorem \cite{Fi} says that, modulo finite-index and compact factors, a Kontsevich--Zorich monodromy $\mathcal{M}_V$ has Zariski closure $\textrm{Sp}(V)$, $\textrm{SU}(p,q)$, $\textrm{SO}^*(2n)$, $\wedge^k \textrm{SU}(p,1)$ or some spin groups. Thus, \emph{all} matrices in $\mathcal{M}_V$ have two eigenvalues with the same modulus \emph{unless} the Zariski closure of $\mathcal{M}_V$ is isomorphic to $\textrm{Sp}(V)$ modulo finite-index and compact factors. It follows that if $\mathcal{M}_V$ is pinching, then $\mathcal{M}_V$ is Zariski dense in $\textrm{Sp}(V)$ modulo finite-index and compact factors. 
\end{remark}

\end{document}